\documentclass[11pt,a4paper]{article}

\usepackage{array}
\usepackage{multirow}

\makeatletter
\newcommand{\thickhline}{%
    \noalign {\ifnum 0=`}\fi \hrule height 1pt
    \futurelet \reserved@a \@xhline
}
\newcolumntype{"}{@{\hskip\tabcolsep\vrule width 1pt\hskip\tabcolsep}}
\makeatother

\renewcommand{\arraystretch}{0.9}

\usepackage{graphics}
\usepackage{epsfig}
\usepackage{epstopdf}
\usepackage{psfrag}
\usepackage{subfigure}
\usepackage{graphicx}
\usepackage{amssymb}
\usepackage{amsthm}
\usepackage{mathrsfs}
\usepackage{hhline}
\usepackage{longtable}
\usepackage{amsmath}
\usepackage{float}
\usepackage{picinpar}
\usepackage{cite}
\usepackage{enumerate}
\usepackage{xcolor}
\usepackage{ifpdf}
\usepackage{caption}

\usepackage[mathlines]{lineno}
\modulolinenumbers[2]
\newcommand*\patchAmsMathEnvironmentForLineno[1]{%
\expandafter\let\csname old#1\expandafter\endcsname\csname #1\endcsname  \expandafter\let\csname oldend#1\expandafter\endcsname\csname end#1\endcsname  \renewenvironment{#1}%
{\linenomath\csname old#1\endcsname}%
{\csname oldend#1\endcsname\endlinenomath}}%
\newcommand*\patchBothAmsMathEnvironmentsForLineno[1]{%
\patchAmsMathEnvironmentForLineno{#1}%
\patchAmsMathEnvironmentForLineno{#1*}}%
\AtBeginDocument{%
\patchBothAmsMathEnvironmentsForLineno{equation}%
\patchBothAmsMathEnvironmentsForLineno{align}%
\patchBothAmsMathEnvironmentsForLineno{flalign}%
\patchBothAmsMathEnvironmentsForLineno{alignat}%
\patchBothAmsMathEnvironmentsForLineno{gather}%
\patchBothAmsMathEnvironmentsForLineno{multline}%
}

\newtheorem{lem}{Lemma}[section]
\newtheorem{thm}[lem]{Theorem}
\newtheorem{cor}[lem]{Corollary}

\newtheorem{pro}[lem]{Proposition}

\usepackage{graphicx}
\graphicspath{%
    {converted_graphics/}
    {/}
}

\title{Independent removable edges in cubic bricks\footnotetext{This work is supported by NSFC (Nos. 12271235, 12361070) and
 NSF of Fujian (Nos. 2021J01978, 2021J06029).} }
\author{
 Fuliang Lu$^{1}$ \ \ and Jianguo Qian$^{2}$\thanks{Corresponding author.\newline E-mail addresses: flianglu@163.com (F. Lu), jgqian@xmu.edu.cn (J. Qian).}
 \unskip\\[.5em]
{\small 1. School of Mathematics and Statistics,
Minnan Normal University, Zhangzhou, PR China}\\
{\small 2. School of Mathematics and Statistics,
Qinghai Minzu University, Xining, PR China}
}
\date{}

\def\r{{removable edge}}

\def\e{{essentially 4-edge-connected}}

\begin{document}

\newcommand{\udots}{\mathinner{\mskip1mu\raise1pt\vbox{\kern7pt\hbox{.}}
\mskip2mu\raise4pt\hbox{.}\mskip2mu\raise7pt\hbox{.}\mskip1mu}}
\maketitle
\begin{abstract}
{ An edge $e$ in a matching covered graph $G$ is {\em removable} if $G-e$ is matching covered, which was introduced by Lov\'asz and Plummer  in connection with ear decompositions of matching covered graphs. A {\it brick}} is a non-bipartite matching covered graph without   non-trivial tight
  cuts.~The importance of bricks stems from the fact   that they are building blocks of matching   covered graphs. Improving  Lov\'asz's result,  Carvalho et al. [Ear decompositions of matching covered graphs, {\em Combinatorica}, 19(2):151-174, 1999] showed that each brick other than $K_4$ and $\overline{C_6}$ has $\Delta-2$ removable edges, where $\Delta$ is the maximum degree of  $G$. In this paper,  we show that every cubic brick $G$ other than $K_4$ and $\overline{C_6}$ has a matching of size at least $|V(G)|/8$, each edge of which is removable in $G$.
\end{abstract}
\vskip 0.2in \baselineskip 0.1in
\section{Introduction}

We consider only undirected simple graphs.   A connected graph $G$ is  {\it matching covered}, {also referred to   as {\it 1-extendable},} if each edge lies in some
perfect matching of $G$.  For the terminologies related to matching covered graphs,
we follow Lov\'asz and Plummer~\cite{LP86}.

For a graph $G$, we denote by $V(G)$ and $E(G)$ the vertex set and edge set of $G$, respectively.
For two disjoint non-empty vertex subsets $X,Y\subseteq V(G)$, we denote by $G[X]$ the
subgraph of $G$ induced by $X$, and by $E_G(X,Y)$ the set of the edges joining one vertex in $X$ and one in
$Y$. In particular, we call $E_G(X,\overline{X})$ an {\it edge cut} of $G$ and denote by {$\partial_G(X)$, or simply by $\partial (X)$,} where $\overline{X}=V(G)\setminus X$. An edge cut $\partial (X)$ is
{\it trivial} if either $|X|=1$ or $|\overline{X}| =1$.
For a non-trivial edge cut $\partial (X)$, we denote the graph obtained from $G$ by contracting  $X$ to a single vertex $x$ by
 $G/X\rightarrow x$, or simply  $G/X$. Further, we call $G/X$ the $\partial (X)$-{\it contraction} of $G$ and $x$ the {\it contracted vertex}.
{An edge cut
$\partial (X)$ of $G$ is  {\it tight }
if $|\partial (X)\cap M|=1$ for every perfect matching $M$ of $G$ and is  {\it separating} if, for any $e\in E(G)$, $G$ has a perfect matching $M_e$ such that $e\in M_e$ and $|\partial (X)\cap M_e|=1$. Obviously, if $G$ is matching covered then every tight cut $\partial (X)$ is separating
and, hence both $G/X$ and  $G/\overline X$ are matching covered.}
 We call a
matching covered graph $G$ that contains no non-trivial tight cuts
a {\it brick} if $G$ is non-bipartite, and a {\it brace}
otherwise. {Edmonds et al.~\cite{ELP8274} } showed that a graph $G$ is a {\it brick} if and only if $G$ is
$3$-connected and $G-\{x,y\}$ has a perfect matching for any two
distinct vertices $x, y\in V(G)$ (bicritical). Further, Lov\'asz~\cite{L87} proved
that any matching covered graph can be decomposed
into a unique list of bricks and braces by a
procedure called the tight cut decomposition.

An edge $e$ of a matching covered graph  $G$ is
{\em removable} if $G-e$ is matching covered.  For $\{e,f\}\subseteq E(G)$, we say that $\{e,f\}$ is a
{\em removable doubleton} of $G$ if neither $e$ nor $f$ is removable and  $G-\{e,f\}$ is matching covered. Removable edge and removable doubleton are generally called {\em removable class}.  The notion of removable class was introduced by Lov\'asz and Plummer, which arises in connection with ear decompositions of
matching covered graphs.

  Lov\'asz \cite{lo} proved  that every brick distinct from $K_4$ and
{$\overline{C_6}$ (the triangular prism) has} a removable edge. Further, Carvalho et al. \cite{CLM9914} showed that each brick $G$ other than $K_4$ and $\overline{C_6}$ has { at least $\Delta-2$ removable edges and the lower bound} is attained by the cubic brick $R_8$ as shown Figure \ref{fig:1}(a), where $\Delta$ is the maximum degree of  $G$. In \cite{zlg} Zhai et al. proved that   the number of removable ears in
 every matching covered graph $G$ is not less than
 the minimum number of the perfect matchings needed to cover all edges of $G$.
Carvalho and  Little  \cite{CL14} showed that every matching covered graph, except the cycle,  has at least three removable classes.

In this paper we consider the number of pair-wise non-adjacent removable edges in  cubic  bricks in terms of the number of vertices.

\begin{thm}\label{thm:main}
Let  $G$ be a  cubic brick other than $K_4$ and $\overline{C_6}$. Then  $G$ has a
matching of size at least $|V(G)|/8$, each edge of which is removable in $G$.

\end{thm}

\begin{figure}[htbp]
 \centering
 \includegraphics[height=2.5cm]{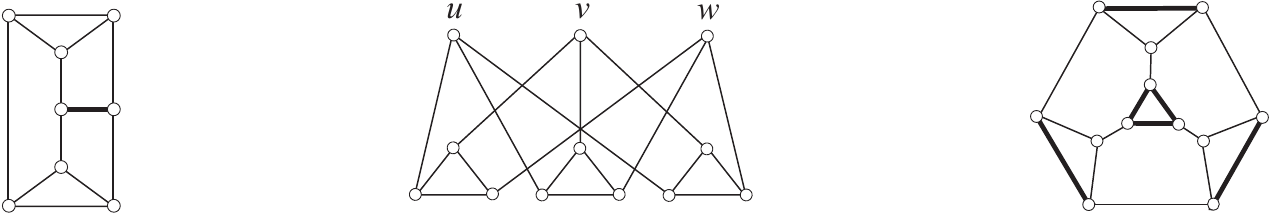}\\
 (a)~~~~~~~~~~~~~~~~~~~~~~~~~~~~~~~~~~~~~~~~~~(b)~~~~~~~~~~~~~~~~~~~~~~~~~~~~~~~~~~~~~~~~~~~~~~~~~~(c)~~~~
\caption{The bold edges represent the removable edges. }
\end{figure}\label{fig:1}

We note that the size $|V(G)|/8$ in Theorem \ref{thm:main} is attained by the graph $R_8$. Further, though Theorem \ref{thm:main}  gives the least numbers of  removable edges for  cubic bricks, this is not the case in general cubic matching covered graphs. For example, the graph shown in Figure \ref{fig:1}(b) is a  3-connected cubic matching covered graph, which  is not  bicritical (the removal of any two vertices in $\{u,v,w\}$,  the left graph does not have a perfect matching) and {hence}, contains no removable edges at all.

{In the following section, we present some basic properties concerning removable edges. In Section 3, we give a proof of Theorem \ref{thm:main}.

\section{ Preliminaries }
In this section, we { recall some known results and present some basic properties concerning removable edges  that will be used in our proof of the main result.}

\subsection{Removable classes }
\label{sec:pr}

Let $G$ be a connected graph with a perfect matching. A nonempty subset $X$ of $V (G)$ is a {\em barrier} if  $o(G-X) = |X|$, where $o(G-X)$ denotes the number of odd components of $G-X$. It follows from the well-known Tutte's Perfect Matching Theorem that  if  $uv\in E(G)$ and $G$ has a barrier that contains both  $u$ and $v$, then no perfect matchings of $G$ contain  the edge $uv$. The following result is also a consequence of Tutte's Theorem.

\begin{lem} {\em  \cite{P9130}} \label{lem:2e}
Every 2-edge-connected cubic graph is matching covered.
\end{lem}

 An edge cut with $k$ edges is called a {\it $k$-cut}. The following proposition follows directly from Lemma \ref{lem:2e}.

\begin{pro}{\em \cite{CLM051}}\label{pro:3cut} Every 3-cut of a  2-edge-connected cubic graph is a separating cut.
\end{pro}

For a matching covered bipartite graph, we have
the following theorem.

\begin{thm}{\em(Theorem 4.1.1 in  \cite{LP86})} \label{thm:bi}
 Let  $G$ be a matching covered bipartite graph with color classes $A$ and $B$. Then $G-\{u,v\}$ has a perfect matching for any $u\in A, v\in B$.
\end{thm}

For two edges $e$ and $f$ in $G$, we say $e$ {\it depends on} $f$ if every perfect
matching of $G$ that contains $e$ also contains $f$. Equivalently, $e$ depends on $f$ if $e$ does not lie in any perfect matching in
$G-f$. We note that an edge $e$ in a matching covered graph $G$ is removable if and only if $e$ is not depended by any other edge of $G$.
We {call an edge cut of a graph $G$ {\em good} if it is separating but not tight, and call a vertex {\em covered by a matching} $M$ if it is incident with some edge in $M$. We note that if $G$ has  a good edge cut, then $G$ has a perfect matching that contains at least three edges in this cut.} Then every separating cut in a brick is good.

\begin{lem}\label{pro:easy}
Let $C=\partial (X)$ be a  good 3-cut of a matching covered graph $G$, $e_0\in E( G[{X}])$, and
$e_1$ an edge in  $C$  that is not depended by any edge in   $E( G[\overline{X}])$ (if such $e_1$ exists).
 Assume that, for any $e\in  C\cup E(G[{X}])\setminus\{e_0\}$,  {$G-e_0$ has  a matching $M$ that contains $e$ and covers every vertex in $X$  and, in particular, $M\cap C=\{e\}$ if $e\in  C\setminus \{e_1\}$.
Then $e_0$ is removable in $G$.}
\end{lem}
\begin{proof}{It suffices to show that, for any $e\in G-e_0$,  $G-e_0$ has a perfect matching containing $e$.   Assume firstly that $e\in E(G[\overline{X}])\cup C\setminus \{e_1\}$. Since $G/X$ is matching covered, it is clear that  $e_1$ is removable in $G/X$. Therefore, $G/X$ has a perfect matching $M_1$ such that $e\in M_1$ and
  $|(C\setminus \{e_1\})\cap M_1|=1$.}
 By hypothesis,  $G-e_0$ has a matching $M'$  such that $M'$ covers every vertex in $X$ and
   $M_1\cap C= M'\cap C$ (noticing $|M_1\cap C|=1$).
  Then  $ (M_1 \cap E(G[\overline{X}]))\cup (M'\cap(E(G[{X}])\cup C))$ is a perfect matching of $G-e_0$  containing $e$.

  {Assume secondly that $e\in E(G[{X}])\cup \{e_1\}\setminus\{e_0\}$. Let $M''$ be a  matching  of  $G-e_0$  that contains $e$ and covers  every vertex in $X$.} Then $|M''\cap C|=1~\mbox{or} ~ 3$.
So
$( M_2\cap E(G[\overline{X}])) \cup (M''\cap(E(G[{X}])\cup C) )$ is a perfect matching of $G-e_0$  containing $e$, where $M_2$ is  a  perfect matching  of $G$  satisfying $M_2\cap C=M''\cap C$ (such $M_2$ exists since  $C$ is a good 3-cut).  So the result follows.
\end{proof}

{As examples for Lemma \ref{pro:easy}, we can verify that all bold edges in Figures \ref{fig:2}-\ref{fig:265} satisfy the condition of the proposition and, hence removable (the graph in each figure is the subgraph induced by $G[X]$, and we assume that every edge in   $\partial (X)$ is  depended by some edge in   $E( G[\overline{X}])$, where $\partial (X)$ is a good 3-cut of $G$). On the other hand,  let's consider  the graph $G$ as shown in Figure \ref{fig:1}(c). We can see that every edge in $G[{X}]$ is removable but does not admit the condition of  Lemma \ref{pro:easy}, where $X$ is the set of the three vertices in the central triangle of $G$.}

The following corollary follows  directly by Lemma \ref{pro:easy}.
\begin{cor}\label{cor:ands}
Let $u_1u_2u_3$ be a triangle of a  cubic brick $G$ and  $u_1v_1\in E(G)$.   If $uv_1$ is removable in $G/\{u_1,u_2,u_3\}\rightarrow u$ and $\partial(\{u_1,u_2,u_3\})$ is good, then $u_2u_3$ is removable in $G$.

\end{cor}

\subsection{Essentially 4-edge-connected cubic graphs}
\label{sec:pf}

A cubic graph is {\em essentially 4-edge-connected} if it is 2-edge-connected and  free of non-trivial 3-cuts.
Kothari et al. showed the following theorems.

\begin{thm}\label{thm:4ecb} {\em \cite{KCLL19a}} Every essentially 4-edge-connected cubic graph is either a brick or a brace.
\end{thm}
\begin{thm}\label{thm:4ecr}{\em \cite{KCLL19a}}
In an essentially 4-edge-connected cubic brick, each edge is either removable or  lies in a removable doubleton.
\end{thm}

For the removability of edges in a brace, we have the following.

\begin{thm}\label{thm:br}{\em \cite{CLM15}}
  each edge in a brace with at least 6 vertices is   removable.
\end{thm}

{For a bipartite graph $G(A,B)$, as usual we also call $A$ and $B$ the color classes of $G$. Further,  if $|A|=|B|$ then we call  $G(A,B)$ {\em balanced}.}
\begin{pro}\label{pro:4econ}
{Let $G$ be an  essentially 4-edge-connected cubic brick other than $K_4$, ${\mathscr{E}_0}$ the collection of all the removable doubletons of $G$, and  $E_0$ the set of the edges in the removable doubletons of ${\mathscr{E}_0}$.} Then  the following statements hold.\\
{\rm(i).} {\em{\cite{LK}}} If  $|{\mathscr{E}_0}|\geq 2$, then $G$ can be decomposed into balanced bipartite vertex-induced subgraphs $G_i$ $(i=1,2,\ldots,|{\mathscr{E}_0}|)$ satisfying
  $E_G(V(G_j),V(G_k))$   is a  removable doubleton of $G$ if $|j-k|\equiv 1\pmod {|{\mathscr{E}_0}|}$ {and $E_G(V(G_j),V(G_k))=\emptyset$ otherwise.}\\
{\rm(ii).} {$G$ has  a perfect matching $M$ such that $M\cap {E}_0=\emptyset$.}
\end{pro}

{\em   Proof of} (ii).
Let $s=|\mathscr{E}_0|$.
We first consider the case when $s=1$. Assume that $\{uv,xy\}$ is the only removable doubleton of $G$.
Since $G-\{uv,xy\}$ is matching covered, the result follows directly by choosing any perfect matching $M$ in $G-\{uv,xy\}$.

We now consider the case  when $s>1$. By (i), we may
assume that $\{x_1u_s,y_1v_s\}\cup(\cup_{i=1 }^{ s-1}\{u_iy_{i+1},v_ix_{i+1}\})$ are all the removable doubletons of $G$, where $\{u_i,y_{i},v_i,x_{i}\}\subset V(G_i)$, and  $u_i$ and $x_{i}$ lie in the same color class of $G_i$, other than the one containing $y_{i}$ and $v_i$ (note that $G_i$ is bipartite), for $i=1,2,\dots,s$. It should be noted that $x_i=u_i$ and $y_i=v_i$ if $|V(G_i)|=2$.

By Theorem \ref{thm:bi},
for any $i\in\{1,2,\ldots,s-1\}$, $G-\{u_iy_{i+1},v_ix_{i+1}\}-\{u_{i-1},v_{i-1}\}$ has a perfect matching (the subscript is taken modulo $s$);
$G-\{x_1u_s,y_1v_s\}-\{u_{s-1},v_{s-1}\}$ has a perfect matching, denoted it by $M_i$ ($ i=\{1,2,\ldots,s\} $). Then
$M_i\cap E(G_i)$ is a perfect matching of $G_i$.
 Therefore, the result follows by setting  $M= \bigcup_{i=1}^s(M_i\cap E(G_i))$.
$\hfill\square$

{The following is a direct consequence of Theorem \ref{thm:br} and Proposition \ref{pro:4econ} (ii).
\begin{cor}\label{cor:4ecr}
Let $G$ be an essentially 4-edge-connected cubic graph other than  $K_4$. Then $G$ has a perfect matching consisting of removable edges of $G$.
\end{cor}
\subsection {Splicing and removable edges}

Let $G$ and $H$ be two vertex-disjoint
graphs. Let $u\in V(G)$ and $v\in V(H)$ be two vertices with the same degree and $\partial(\{u\})=\{uu_1,uu_2,\ldots,uu_d\}$ and
$\partial(\{v\})=\{vv_1,vv_2,\ldots,vv_d\}$. The \emph{splicing of $G$
and $H$ at $u$ and $v$}, denote by $G(u)\odot H(v)$ (or simply $G(u)\odot H$ or $G\odot H$ if no confusion occurs), is the graph
obtained from $G-u$ and $H-v$ by adding a new edge between the two
vertices $u_i$ and $v_i$ for each $i=1,2,\ldots,d$. The two vertices
$u$ and $v$ are called the \emph{splicing vertices},    every edge $u_iv_i$ in $\{u_1v_1,u_2v_2,\ldots,u_dv_d\}$ is called the \emph{splicing edge}. In particular, if $H=K_4$ and $u$ is a vertex of degree 3, then the splicing operation  $G(u)\odot K_4$ can be intuitively viewed as the operation that `inserts a triangle' at $u$. In this sense, we also call such an operation the {\it triangle-insertion} at $u$ of $G$ and denote  $G(u)\odot K_4$ simply by $G\langle u\rangle$.

We note that every edge (or vertex) in  $G\odot H$ corresponds uniquely an edge (or vertex other than the splicing vertex) in  $G$ or $H$.  With a mild abuse of language, we will  use the same label of the edge (or vertex) in $G\odot H$ as it is in $G$ or $H$, and vice versa.
The following propositions are about the splicing of two graphs.

\begin{pro}\label{pro:spl1}{\em \cite{CLM051}}
Any graph obtained by splicing two matching covered graphs
is also matching covered.
\end{pro}
\begin{pro}\label{lem:3-sp}{\em\cite{CLM051}}
 Any splicing of two cubic bricks is a cubic brick.
\end{pro}

\begin{pro}\label{pro:e3}
Let $uv$ be a (removable) edge of a matching covered graph $G$ with $d_G(v)=3$. Then no splicing edge in $G\langle v\rangle$  is removable.
\end{pro}
 \begin{proof}
Let $xyz$ be the insertion triangle and $ux$  an arbitrary edge of the three splicing edges in $G\langle v\rangle$. Then $yz$ cannot lie in any perfect matching of $G\langle v\rangle-ux$ since $x$ is isolated in $G\langle v\rangle-ux-\{y,z\}$.
 \end{proof}

 For the removability of an edge that is not a splicing edge, we have the following proposition.
\begin{lem}\label{pro:spl} Let $G_1$ and $G_2$ be matching covered graphs, $u\in V(G_1)$, $v\in V(G_2)$, $d_{G_1}(u)=3$ and $d_{G_2}(v)=3$.
And let $H=G_1(u)\odot G_2(v)$. For any edge $e$ in $G_1$
that is not incident with $u$,
\\
{\rm (i).} if $e$ is removable in $G_1$, then $e$ is removable in $H$; and\\
{\rm (ii).} if $e$ is removable in $G_1\langle u\rangle$ and $\partial_H(V(G_1-u))$ is good in $H$, then $e$ is removable in $H$.
\end{lem}
\begin{proof}
 {(i). Since $G_1-e$ is matching covered and   $d_{G_1-e}(u)=3$, (i) follows directly by Proposition \ref {pro:spl1}.

(ii). Let $\partial_{G_1}(\{u\})=\{ux,uy,uz\},\partial_{G_2}(\{v\})=\{vx',vy',vz'\}$. Let $C'=\{xx',yy',zz'\}$ and  $C''=\{xx'',yy'',zz''\}$ be the set of splicing edges corresponding to the splicing of $G_1(u)\odot G_2(v)$ and  $G_1\langle u\rangle$, respectively, i.e., $\partial_H(V(G_1-u))=\{xx',yy',zz'\}$ and  $\partial_{G_1\langle u\rangle}(V(G_1-u))=\{xx'',yy'',zz''\}$. Let $e_1$ be an arbitrary edge in $H-e$. We will show that $e_1$ lies in some perfect matching of $H-e$ to complete the proof.

We first assume that $e_1\in E(G_1)$. Let $M_{e_1}$ be a perfect matching of $G_1\langle u\rangle-e$ that contains $e_1$. If $M_{e_1}$ contains exactly one edge in $C''$, say $zz''\in M_{e_1}$, then $(M_{e_1}\cap E(G_1))
\cup M_1$ is a perfect matching of $H-e$ that contains $e_1$, where $M_1$ is a perfect matching of $G_2$ containing $vz'$. We now assume that $C''\subseteq M_{e_1}$. Since $C'$ is a good 3-cut in $H$, $H$ has a perfect matching $M_2$ that contains all the edges in $C'$. Therefore, $(M_{e_1}\cap E(G_1))\cup (M_2\cap E(G_2))\cup C'$ is a perfect matching of $H-e$ that contains $e_1$.

We now assume that $e_1\in E(G_2)$. Let $M_3$ be a perfect matching of $G_2$ containing $e_1$.
It is clear that $|M_3\cap \partial_{G_2}(\{v\})|=1$, say $vz'\in M_3$.  Since $G_1\langle u\rangle-e$ is matching covered, $G_1\langle u\rangle-e$ has a perfect matching $M_4$ that contains $x''y''$ and, hence, contains $zz''$ as the only edge in $C''$ since $x''y''z''$ is a triangle of $G_1\langle u\rangle$.
 Then $M_3\cup (M_4\cap E(G_1))$ is a perfect matching of $H-e$ that contains $e_1$.}
\end{proof}

The following corollary will be useful for the case when the two vertices are splicing
vertices in different splicings.

\begin{cor}\label{cor:spl2} Assume that $G_0$, $G_1$ and $G_2$ are matching covered graphs, $u_i\in V(G_i)$ ($i=1,2$). Let $v_1,v_2\in  V(G_0) $ and $G=(G_0(v_1)\odot G_1(u_1))(v_2)\odot G_2(u_2)$. For $i=1$ and $2$, let $E_i$ be the set of edges in $E(G_i)\setminus \partial (\{u_i\})$ that are removable in $G_0(v_i)\odot G_i(u_i)$. Then every edge in $E_1\cup E_2$ (if not empty) is removable in $G$.
\end{cor}
\begin{proof}Note that every edge in $E_1$ is not a splicing edge in  the two splicings.
 By Lemma \ref{pro:spl} (i), every edge in $E_1$ is \r~in $G$.  Similarly, every edge in $E_2$ is \r~in $G$, since
 $(G_0(v_1)\odot G_1(u_1))(v_2)\odot G_2(u_2)= (G_0(v_2)\odot G_2(u_2))(v_1)\odot G_1(u_1)$. So the result follows.
\end{proof}



\subsection {$\Delta$-replacements and removable edges}

 For a 3-cut $\partial (X)$ of a  graph $G$ with $|\overline{X}|\geq 5$, we denote  $G^{\Delta}(X)=(G/\overline{X}\rightarrow x)\langle x\rangle$. We call $G^{\Delta}(X)$  the {\em $\Delta$-replacement}  of $X$ in $G$ and call the triangle inserted at $x$ the {\em replacement-triangle}.
\begin{pro}\label{pro:e}
Let $\partial (X)$ be a 3-cut of {a brick $G$ with $|\overline{X}|\geq 5$.} Then $G^{\Delta}(X)$ is a brick.\end{pro}

 \begin{proof}
{It is clear that $G^{\Delta}(X)$ is 3-connected. We show that $G^{\Delta}(X)$ is bicritical, that is,
 $G^{\Delta}(X)-\{u,v\}$ has a perfect matching for any two vertices $u,v$  in   $G^{\Delta}(X)$.

Let  $x_1x_2x_3$ label the replacement-triangle of  $G^{\Delta}(X)$, and $\partial_{G^{\Delta}(X)}(\{x_1,x_2,x_3\})=\{x_1x'_1,x_2x'_2,x_3x'_3\}$ and $\partial_{G}({X})=\{x_1'x''_1,x_2'x''_2,x_3'x''_3\}$.
Set $w=x_i''$ if $u=x_i$; $w=u$ otherwise. And set $z=x_i''$ if $v=x_i$; $z=v$ otherwise.

 Since $G$ is a brick, $G-\{w,z\}$ has a perfect matching, say $M$. Set $M'=(M\cap G[X])\cup (M \cap \partial_G (X)) $. Since $|X|$ is odd,
 it can be checked that  $M'$ covers all  vertices in  $G^{\Delta}(X)-\{u,v\}$, or  all  vertices in  $G^{\Delta}(X)-\{u,v\}$ except exactly two vertices in $\{x_1,x_2,x_3\}$. In the latter case, $M'$, together with the edge between the two vertices that are not covered by $M'$,  is a perfect matching of $G^{\Delta}(X)-\{u,v\}$
 (note that  any two vertices in $\{x_1,x_2,x_3\}$ has an edge).
}
 \end{proof}

\begin{pro}\label{pro:e2} Let $G$ be a matching covered graph, and $\partial (X)$ be a good  3-cut of $G$ such that $|\overline{X}|\geq 5$.
 If $ e\in E(G[X]) $ and $e$ is removable in $G^{\Delta}(X)$,  then $e$ is removable in $G$.
 \end{pro}

 \begin{proof}Let $G_1=G/\overline{X}\rightarrow x$. Then $G^{\Delta}(X)=G_1\langle x \rangle$. By Lemma \ref{pro:spl}, the result follows.
 \end{proof}

\section{Proof of Theorem \ref{thm:main}}
\label{sec:non-bi}

{In this section, we consider only cubic graphs. As an extension of triangle-insertion, for vertices $v_1,v_2,\ldots,v_k$ of $V(G)$, we denote the graph obtained from $G$ by the triangle-insertions at  $v_1,v_2,\ldots,v_k$ by $G\langle v_1,v_2,\ldots,v_k\rangle$.}
For non-trivial 3-cuts $\partial (X_1), \partial (X_2),\ldots, \partial (X_t) $ $(t\geq 1)$ of $G$,  $G'$ is the graph  obtained from $G$ by contracting $X_1,X_2,\ldots,X_t$ to $t$ vertices $x_1,x_2,\ldots,x_t$, respectively, where $X_i\cap X_j=\emptyset$ ($ i\neq j$, $i,j\in [1,2,\ldots, t]$).  For convenience, we denote $X_1\cup X_2\cup \cdots\cup X_t$ by
$\{x_1,x_2,\ldots,x_t\}^{-1}_G$, or simply $\{x_1,x_2,\ldots,x_t\}^{-1}$ without confusion. It can be seen that $\{x_1,x_2,\ldots,x_t\}^{-1}_G=V(G)\setminus V(G')$.

Let $\mathscr{G}$ be the family of $K_4$ and the cubic graphs obtained from $K_4$ by a sequence of successive triangle-insertions. It is clear that every graph $G$ in $\mathscr{G}$ is a cubic brick by Proposition \ref{lem:3-sp}, and except $K_4$, every vertex in $G$ lies in at most one triangle.  Moreover, for $G\in\mathscr{G}$ and $G\not=K_4$, the graph  obtained from $G$ by contracting $k$ ($k\geq 1$)  triangles  to $k$ vertices, respectively, is still in $\mathscr{G}$.

\begin{lem}\label{lem:nnb1} Let $G\in \mathscr{G}$ and let $uvw$ be a triangle of $G$. If  $|V(G)|\geq 10$, then
$G$ has a  non-trivial 3-cut $\partial({Y})$ such that $|Y|\geq 5$,  $\{u,v,w\} \cap  Y=\emptyset$ and $G[Y]$
has a matching of size at least $ (|Y|-3)/8+1$, each edge of which is removable in $G$.
\end{lem}
\begin{proof} Since $G\in \mathscr{G}$, $G$ has at least two disjoint triangles. Let $G_1$ be the graph obtained from $G$ by contracting all the triangles of $G$ other than  $uvw$,  respectively. Let $S_1$ be the set of the contracted vertices of $G_1$.

We first claim that if $G_1$ has a triangle other than $uvw$, say $x_1y_1z_1$, that contains more than one vertices in $S_1$, then the lemma holds.

 Let $Y_1=\{x_1,y_1,z_1\}^{-1}_G$. We consider the following two cases.

Case 1.  $\{x_1,y_1,z_1\}\subseteq S_1$.

Since  $x_1y_1z_1$ is a triangle in $G_1$, $G[Y_1]\cong G^*_3$, see Figure \ref{fig:2}.  It can be checked that each bold edge in $G^*_3$ is removable by Lemma \ref{pro:easy}. Choosing an arbitrary edge from each triangle of $G^*_3$, respectively, we get a matching $M$. Since
$|M|=3>6/8+1=(|Y_1|-3)/8+1$ and $Y_1\cap\{u,v,w\}=\emptyset$, $Y_1$ and  $M$   satisfy the requirement of the lemma.

\begin{figure} [htbp]
\begin{center}

\includegraphics[width=10cm]{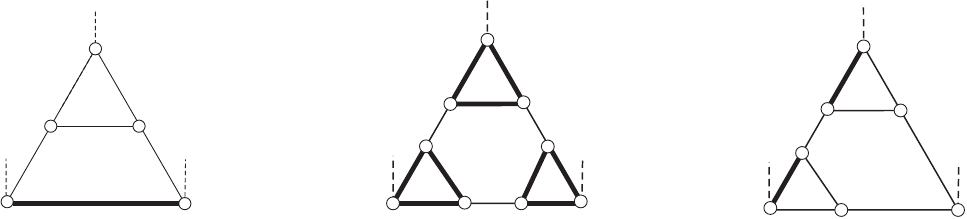}\\
$G^*_{2}$~~~~~~~~~~~~~~~~~~~~~~~~~~~$G^*_{3}$~~~~~~~~~~~~~~~~~~~~~~~~~~~~$G^*_{4}$
\end{center}
\caption{The bold edges represent the removable edges. }\label{fig:2}
\end{figure}

Case 2.  $|\{x_1,y_1,z_1\}\cap S_1|=2$.

In this case,   $G[Y_1]\cong G^*_{4}$ (Figure \ref{fig:2}). Let $M$ be the matching consisting of the two bold edges in $G^*_4$. Again by Lemma \ref{pro:easy}, every edge in $M$ is removable in $G$.  Since $2 >4/8+1=(|Y_1|-3)/8+1$, $Y_1$ and $M$ are as desired. Our claim follows.

Now we assume  that every triangle  of $G_1$, distinct from $uvw$,  contains exactly one vertex in $S_1$. This implies that $G[\{x,y,z\}^{-1}]\cong G_2^*$ for any triangle $xyz$ of $G_1$ other than $uvw$. Let $G_2$ be the graph obtained from $G_1$ by  contracting all the triangles   other than $uvw$ and let $S_2$ be the set of the contracted vertices of $G_2$. It is clear that $G_2\in \mathscr{G}$ and $S_2\cap S_1=\emptyset$. Further, every triangle in $G_2$ other than $uvw$ contains at least one vertex in $S_2$.
Recalling that $|V(G)|\geq 10$,  we have $|V(G_2)|\geq 6$.
So we can choose a triangle of $G_2$ other than $uvw$, say $ x_2y_2z_2$, and let $Y_2=(\{x_2,y_2,z_2\}^{-1}_{G_1})_G^{-1}$.

Case A.    $\{x_2,y_2,z_2\}\subseteq S_2$.

In this case,  $G[Y_2]\cong G^*_3\langle u',v',w'\rangle$, where  $u',v',w'$ are three vertices in different triangles of $G^*_3$, respectively.
Recall that $G^*_3$  has   9  removable edges. So by Corollary \ref{cor:ands}, $G^*_3\langle u',v',w'\rangle$ has a matching of size at least 6\ ($>12/8+1=(|Y_2|-3)/8+1$), every edge of which is removable.

Case B.    $|\{x_2,y_2,z_2\}\cap S_2|= 2$.

 If  $|\{x_2,y_2,z_2\}\cap S_1|=1$, then $G[Y_2]\cong G^*_3\langle u',v'\rangle$, where $u'$ and $v'$ are two vertices in different triangles of $G^*_3$, respectively. By Corollary \ref{cor:ands}, $G^*_3\langle u',v'\rangle$ has a matching of size 5\ ($>10/8+1=(|Y_2|-3)/8+1$), every edge of which is  removable. If $|\{x_2,y_2,z_2\}\cap S_1|=0$, then $G[Y_2]$ is isomorphic to  one of the graphs as listed in Figure \ref{fig:263}. Further, every graph in  Figure \ref{fig:263} has a matching of size 2 ($=8/8+1=(|Y_2|-3)/8+1$) (consisting of the  bold edges), each edge of which is  removable in $G$.
\begin{figure} [htbp]
\begin{center}
\includegraphics[width=8cm]{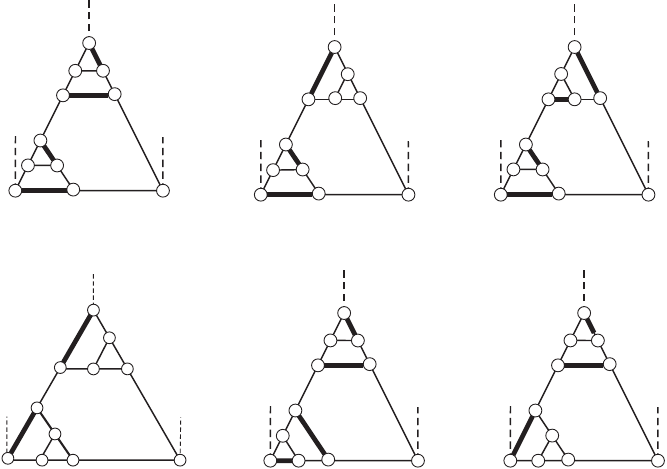}
\end{center}
\caption{ $G[Y_2]$ when $|\{x_2,y_2,z_2\}\cap S_2|= 2$ and $|\{x_2,y_2,z_2\}\cap S_1|=0$.   }\label{fig:263}
\end{figure}

Case C.  $|\{x_2,y_2,z_2\}\cap S_2|= 1$.

 Without loss of generality, assume that $z_2\in S_2$. Let $r_{\max}$ be the maximum size of a matching consisting only of removable edges in $G[Y_2]$. We list the lower bounds of $r_{\max}$ and the corresponding graph $G[Y_2]$ by distinguishing the three cases, i.e.,  $|\{x_2,y_2\}\cap S_1|=0,1,2$, as the following table. The table shows that $G[Y_2]$ has a matching of size at least $(|Y_2|-3)/8+1$, each edge of which is removable in $G$.\end{proof}

\renewcommand{\arraystretch}{1.6}
\begin{table}[h]\label{Ta1}
 \fontsize{10}{8}\selectfont
\caption{The lower bounds of $r_{\max}$ and the subgraph $G[Y_2]$ when $|\{x_2,y_2,z_2\}\cap S_2|= 1$.}
\begin{center}
\begin{tabular}{|c|c|c|c|} 
\hline
Cases & $|Y_2|$ & $r_{\max}$ & $G[Y_2]$\\ 
\hline 
$\{x_2,y_2\}\cap S_1=\emptyset$   & 7&$\geq$ 2 & Figure \ref{fig:264}
\\ \hline
 $|\{x_2,y_2\}\cap S_1|=1$ &9 &  $\geq$ 2 & Figure \ref{fig:265} \\
\hline
 $|\{x_2,y_2\}\cap S_1|=2$ &11 & $\geq$ 4  & $G^*_3\langle v\rangle, v\in V(G^*_3)$ \\
\hline
\end{tabular}
\end{center}
\end{table}
\begin{figure} [htbp]
\centering
\includegraphics[width=6cm]{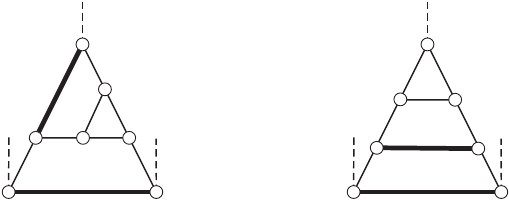}
\caption{ $G[Y_2]$ when $|\{x_2,y_2,z_2\}\cap S_2|= 1$ and  $\{x_2,y_2\}\cap S_1=\emptyset$. }\label{fig:264}
\end{figure}
\begin{figure} [htbp]
\centering
\includegraphics[width=8cm]{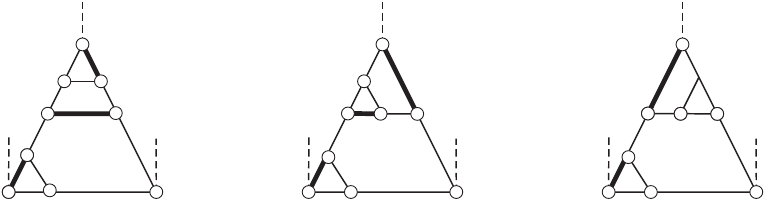}
\caption{ $G[Y_2]$ when $|\{x_2,y_2,z_2\}\cap S_2|= 1$ and  $|\{x_2,y_2\}\cap S_1|=1$.}\label{fig:265}
\end{figure}

We say that
two edge cuts $\partial(X)$ and $\partial(Y)$  {\em cross} if the four sets ${X} \cap Y$, $\overline{X} \cap Y$, ${X} \cap \overline{Y}$ and $\overline{X} \cap \overline{Y}$ are all nonempty.
\begin{lem}\label{lem:brace2}Let $G_0[A,B]$ be a cubic brace,  $u\in A$ and
$v\in B$.  And let $G_1$ and $G_2$ are two cubic bricks, $G=(G_0(u)\odot G_1)(v)\odot G_2$. Then $G$ is a brick.
 \end{lem}
\begin{proof}  Obviously, $G$ is 3-connected cubic matching covered graph.
Suppose, to the contrary, that $\partial (X)$ is a tight cut of $G$.
By Theorem 8 in \cite{KCLL19a}, $\partial (X)$ is a 3-cut. Let $Y=V(G)\cap V(G_1)$. Then $\partial (Y)$ is a 3-cut of $G$. We claim that $\partial (X)$   and $\partial (Y)$ do not cross.

 Let $\partial(X)=\{x_i\overline{x_i}:i=1,2,3\}$,
 $\partial(Y)=\{y_i\overline{y_i}:i=1,2,3\}$ such that $x_i\in X$, $y_i\in Y$ for $i=1,2,3$. We claim that
 $G[X]$ is connected. Otherwise, suppose $G_1$ is one of the  components of $G[X]$. Then one of $G_1$ and $G[X\setminus V(G_1)]$, say $G_1$, contains at most one vertex in $\{x_1,x_2,x_3\}$ by pigeonhole principle. Since $G$ is connected, at least one vertex in $\{x_1,x_2,x_3\}$ lie in $G_1$. So $|\{ x_1,x_2,x_3 \}\cap V(G_1)|=1$.
  Without loss of generality, assume that $x_1\in V(G_1)$.
  Then $G- \overline{x_1}$ is not connected, contradicting  the fact that $G$ is 3-connected. Similarly, $G[\overline{X}]$, $G[Y]$ and $G[\overline{Y}]$ are connected.

Suppose to the contrary that $\partial(X)$ and $\partial(Y)$  cross. We will show that $G$ is not 3-connected, which derives a contradiction.
Recalling  $G[\overline{X}]$ and $G[X]$ are connected, and $X\cap Y$ and
$X\cap \overline{Y}$ are the partition of $X$, we have $\{ x_1,x_2,x_3 \}\cap Y\neq \emptyset$ and $\{ x_1,x_2,x_3 \}\cap \overline{Y}\neq \emptyset$.
 Without loss of generality, assume that $\{x_1, x_3\}\subset X\cap Y$ and ${x_2}\in X\cap \overline{Y}$. Similarly, assume that
  ${y_1}\in X\cap Y$,
  $\overline{y_1}\in X\cap \overline{Y}$, $\{\overline{x_1}, {y_2}\}\subset \overline{X}\cap Y$ and $\{\overline{x_2}, \overline{y_2}\}\subset \overline{X}\cap \overline{Y}$.
If $\overline{x_3}\in \overline{X}\cap \overline{Y}$, then $x_3\overline{x_3}=y_3\overline{y_3}$. Therefore, $G[ X\cap \overline{Y}]$ is a component of $G-\{y_1,\overline{x_2}\}$.
Now we assume that $\overline{x_3}\in \overline{X}\cap {Y}$. Then  $G[ {X}\cap \overline{Y}]$ is a component of $G-\{y_1,\overline{x_2}\}$ when $\{\overline{y_3}, {y_3}\}\subset \overline{X}$;  $G[ \overline{X}\cap \overline{Y}]$ is a component of $G-\{x_2, y_2\}$ when $\{\overline{y_3}, {y_3}\}\subset {X}$. So the claim follows.

 Similarly, $\partial (X)$ and $\partial (Z)$ do not cross, where $Z=V(G)\cap V(G_2)$.
 Therefore, $\partial (X)$ is  a subset of $E(G_0)$, $E(G_1)$ or $E(G_2)$. Suppose, without loss of generality, that $\partial (X)\subset E(G_0)$. Then $\partial (X)$ is a tight cut of $G_0$, contradicting the fact that $G_0$ is a brace.
Therefore, $G$ is a brick.
\end{proof}

 For a graph $G\notin  \mathscr{G}$, we say $X$ is the {\em minority} of a 3-cut $\partial (X)$ if   $G/\overline X\in  \mathscr{G}$.
A non-trivial 3-cut $\partial (X)$   is {\em maximal} if there does not exist any non-trivial 3-cut $ \partial (Y)$ such that $G/\overline Y\in  \mathscr{G}$ and $X\subset Y$, where $X$ is the { minority} of $\partial (X)$.

\begin{lem}\label{lem:nnb2}Let $G$ be a cubic brick  and $G\notin\mathscr{G}$. If for  every maximal 3-cut $\partial (X)$ of $G$, $|X|\leq 5$, where $X$ is the {minority} of $\partial (X)$, or $G$ has a  non-trivial 3-cut $C$ such that one of the $C$-contractions  is essentially 4-edge-connected other than $K_4$, then
 $G$ has a matching $M$ consisting of the removable edges of $G$ such that\\
   {\em (i).}  $|M|\geq |V(G)|/6$ or,\\
    {\em (ii).}
    $M\subset E( G[Y] )$ and $|M|\geq (|Y|-3)/8+1$, where $\partial (Y)$ is a non-trivial 3-cut of $G$ such that $|Y|\geq 5$.
 \end{lem}
\begin{proof} By
 Corollary \ref{cor:4ecr}, $G$ has a perfect matching consisting of removable edges when $G$ does not  have non-trivial 3-cuts. Then
 we may assume that $G$ has a non-trivial 3-cut.
Assume first that
 $G$ has a  non-trivial 3-cut $\partial({X}) $ such that $G/\overline{X}\rightarrow \overline{x}$ is essentially 4-edge-connected distinct from $K_4$.
By Corollary \ref{cor:4ecr}, $G/\overline{X}$ has a perfect matching $M_1$ such that every edge of $M_1$ is removable in $G/\overline{X}$.
Assume that $\overline{x}u \in M_1$.
By Lemma \ref{pro:spl}, every edge in $M_1\setminus\{ \overline{x}u \}$ is removable in $G$. Since $|M_1\setminus\{ \overline{x}u \}|={|V(G/\overline{X})|}/{2}-1=(|X|-1)/{2} >{(|X|-3)}/{8}+1$
(note that $|V(G/\overline{X})|\geq 6$), the result holds in this case.

 Now we assume that, for every  non-trivial 3-cut $C$,  the $C$-contractions are $K_4$,  or  not essentially 4-edge-connected. Since $G\notin\mathscr{G}$,   in every non-trivial 3-cut $C$, at most one of the $C$-contractions belongs to  $\mathscr{G}$.
 By hypothesis,
 every maximal 3-cut $\partial (X)$ of $G$ satisfies $|X|\leq 5$, where $X$ is the {minority} of $\partial (X)$.
 Contracting the minorities of
 all the maximal 3-cuts  of  $G$   to single vertices, respectively, we get $G'$.
 Now we consider two cases depending on whether $G'$ is essentially 4-edge-connected.

 Case. 1.  $G'$ is not essentially 4-edge-connected.

Choose a non-trivial 3-cut $\partial (X_1)$ of $G'$ such that every 3-cut $\partial (X_1')$ satisfying $X_1'\subset X_1$ is trivial.
Then $H=G'/ \overline{X}_1\rightarrow x $ is essentially 4-edge-connected other than $K_4$. Moreover, $H\notin\mathscr{G}$. Let $G_1^*$ be a triangle, and   $S_i=\{u\in X_1| G[\{u\}_G^{-1}]\cong G_i^*  \}$ for $i=1,2$,
  and let $S=S_1\cup S_2$, $s_i=|S_i|$.
  Let $X_2=X_1\cup S_G^{-1}\setminus S$.
  Then $X_2\subset V(G)$, $\partial (X_2)$ is also a 3-cut of $G$ and
   $|X_2|=|V(H)|+2s_1+4s_2-1$.
  For $u\in S$, let $Y_u=\{u\}_G^{-1}$. Let $M_u$ be a maximum matching of $G[Y_u]$ such that  every edge in it is removable in $G$.

  If $H$ is not bipartite, then it is a brick by Theorem \ref{thm:4ecb}. Assume that $u\in S$, $G_u= G/\overline{Y_u}\rightarrow u'$ and $H'=H(u)\odot G_u(u')$. Then $H'$ is a brick by Proposition \ref{lem:3-sp}. So $\partial ( Y_u )$ is good in $H'$.  By Corollary \ref{cor:ands},
 the subgraph $H'[Y_u]$  contains one removable edges of $H'$ when $u\in S_1$. By Corollary \ref{cor:spl2} repeatedly,
   the subgraph $G[Y_u]$  contains one removable edge of $G$ when $u\in S_1$.
    Similarly, $ |M_u|>1$ when $u\in S_2$, as $G_2^*$ contains one removable edge (see Figure \ref{fig:2}).

 If $H$ is bipartite, then it is a brace by Theorem \ref{thm:4ecb} again.
  Each color class of $H$ contains some vertex in $S\cup \{x\}$. Otherwise, the color class that contains no vertex in $S\cup \{x\}$ is a barrier of $G$, which contradicts  the fact that $G$ is a brick.
 Assume that $v\in S$, $u$ and $v$ lie in different color classes of $H$ and $H''=(H(u)\odot G_u(u'))(v)\odot G_v(v')$, where $G_v= G/\overline{Y_v}\rightarrow v'$. Then $H''$ is a brick by Lemma \ref{lem:brace2}.  So $\partial ( Y_u )$ and $\partial ( Y_v )$ are good in $H''$. With the same reason as  last paragraph, we have $ |M_u|>1$ when $u\in S$.

By Corollary \ref{cor:4ecr}, $H$ has a perfect matching $M_2$ consisting of removable edges in $H$.
If $e\in M_2$ and $e$ is  incident with some vertex in $S\cup \{x\} $, then $e$ may not be removable in $G$. Let $M_2' \subset M_2$ such that every edge in $M_2'$ is not incident with any vertex in $S\cup \{x\}$ (When $|S\cup \{x\}|\geq{|V(H)|}/{2}$,  $M_2' $ may be empty). Let $M=(\bigcup_{u\in S} M_u)\cup M_2'$.
Then every edge in $M$ is removable in $G$ by Lemma \ref{pro:spl} and Corollary \ref{cor:spl2}. Moreover, $M$ is a matching.

 We now estimate the lower bound of $|M|$.  Firstly, assume that $ |V(H)|\geq 8$.
We consider the case when $|S\cup \{x\}|\geq |V(H)|/2$. Note that every vertex in $H$ is incident with one edge in $M_2$. So $ \sum_{u\in S} |M_u|\geq  s_1+s_2. $
Then
$\frac{ |M|-1}{|X_2|-3}\geq
\frac{s_1+s_2-1 }{ |V(H)|-4+2s_1+4s_2 }$. If $\frac{s_1+s_2-1 }{ |V(H)|-4+2s_1+4s_2 }\geq\frac{1 }{ 2 }$, then the result follows. So we assume that   $\frac{s_1+s_2-1 }{ |V(H)|-4+2s_1+4s_2 }<\frac{1 }{ 2 }$. Then
$\frac{s_1+s_2-1 }{ |V(H)|-4+2s_1+4s_2 }\geq
\frac{s_2-1 }{ |V(H)|-4+4s_2 }\geq
\frac{|V(H)|/2-2}{ |V(H)|-4+4(|V(H)|/2-1) }\geq\frac{|V(H)|/2-2 }{3|V(H)|-8}\geq\frac{2}{16}=\frac{1}{8}~(\mbox{as}~s_2\geq |V(H)|/2-1~\mbox{and}~ |V(H)|\geq 8).
$ So $|M|\geq (|X_2|-3)/8+1$. The result follows.

We consider the case when $|S\cup \{x\}|\leq |V(H)|/2-1$.
Recalling  $|M_2'|\geq |V(H)|/2-1- s_1-s_2$ and $|M|=|M_2'|+\sum_{u\in S} |M_u|\geq |V(H)|/2-1$, we have
$$\hspace{-3cm}
\frac{ |M|-1}{|X_2|-3}\geq
\frac{|V(H)|/2-2 }{ |V(H)|-4+2s_1+4s_2 }\geq
\frac{|V(H)|/2-2 }{ |V(H)|-4+4s_2 }\geq
\frac{|V(H)|/2-2 }{ |V(H)|-4+4(|V(H)|/2-2) }$$
$$\hspace{-3.5cm}\geq\frac{|V(H)|/2-2 }{3|V(H)|-12}\geq\frac{1}{6}>\frac{1}{8}~(\mbox{as}~s_2\leq |V(H)|/2-2~\mbox{and}~ |V(H)|\geq 8).
$$

Now we assume that $ |V(H)|=6$. Noting that $\overline{C_6}$ is the only cubic brick with 6 vertices, which is not \e, we have $H$ is bipartite, that is $H \cong K_{3,3}$.
 Note that every edge is removable in $H$ by
Theorem \ref{thm:br}. It can be checked that the subgraph $G[\{u\}_G^{-1}]$  contains two non-adjacent removable edges of $G$ when $u\in S_2$ (see Figure \ref{fig:re}(a), in the figure, only the case when the two color classes of $H$   contain  one vertex in $S$ resp. is illustrated, the other cases are similar by Lemma \ref{pro:spl} and Corollary \ref{cor:spl2}).
Then we have\\
$
|M|\geq
\begin{cases} s_1+2s_2+3-s_1-s_2,\quad \ \ &
  |S\cup\{x\}|\leq 2\\
 s_1+2s_2, \quad \ \ & |S\cup\{x\}|\geq 3
\end{cases},$ and $|X_2|= 5+2s_1+4s_2$. Therefore,
\[\hspace{-3.5cm}
\frac{ |M|-1}{|X_2|-3}\geq
\begin{cases}\frac{s_2+2}{ 2+2s_1+4s_2}\geq \frac{2}{ 2+2s_1} \geq \frac{1}{ 2},\quad \ \ &
  |S\cup\{x\}|\leq 2\\
 \frac{s_1+2s_2-1}{2+2s_1+4s_2}\geq \frac{ s_1-1}{ 2s_1+2}\geq  \frac{1}{ 6}, \quad \ \ & |S\cup\{x\}|\geq 3
\end{cases}. \mbox{ The result follows in this case.}
\]

\begin{figure} [htbp]
\begin{center}

\includegraphics[width=9cm]{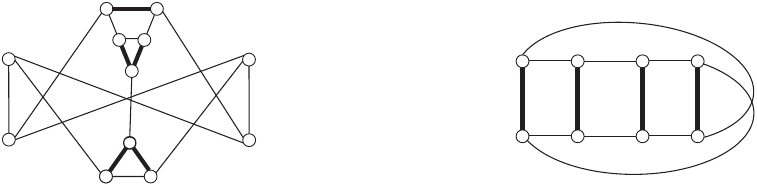}\\ ~~~~~~~~~~~~~~~~~~~~(a)~~~~~~~~~~~~~~~~~~~~~~~~~~~~~~~~~~~~~~~~~~~~(b)~~~~~~~~~~~~~~~~~~~~~
\end{center}
\caption{The bold edges represent the removable edges. }\label{fig:re}
\end{figure}

Case. 2.  $G'$ is   essentially 4-edge-connected.

Similar to Case 1, for $i=1,2$, let
 $S_i=\{u\in S| G[\{u\}_G^{-1}]\cong G_i^*\}$, and $s_i=|S_i|$. Let $S=S_1\cup S_2$, that is $S=V(G')\setminus V(G)$.
   Then
$|V(G)|=|V(G')|+2s_1+4s_2$. Recall that
 $M_u$ is a maximum matching of $G[\{u\}_G^{-1}]$ such that  every edge in it is removable in $G$ for $u\in S$.  Then $\sum_{u\in S} |M_u|\geq  s_1+s_2. $
By Corollary \ref{cor:4ecr}, $G'$ has a perfect matching $M_3$ consisting of removable edges in $G'$. Let $M_3' \subset M_3$ such that every edge in $M_3'$ is not incident with any vertex in $S$. Let $M= M_3'\cup(\cup_{u\in S} M_u)$.

Note that $|V(G')|\geq 6$.
If $s_1+s_2\leq  {|V(G')|}/{2}$, then
 $|M|\geq s_1+s_2+|V(G')|/2- (s_1+s_2)=|V(G')|/2$, therefore
 $ \frac{|M|}{|V(G)|}=\frac{|V(G')|/2}{|V(G')|+2s_1+4s_2}\geq \frac{|V(G')|/2}{|V(G')|+4s_2}\geq \frac{1}{6}~~(\mbox{as~} 1\leq s_2\leq |V(G')|/2).$
If $s_1+s_2>  {|V(G')|}/{2}$, then $|M|\geq s_1+s_2 $, so  $\frac{|M|}{|V(G)|}=\frac{s_1+s_2}{|V(G')|+2s_1+4s_2} \geq \frac{s_2}{|V(G')|+4s_2}\geq \frac{|V(G')|/2+1}{|V(G')|+4(|V(G')|/2+1)}\geq \frac{2}{11}~~(\mbox{as~} |V(G')|\geq s_2\geq |V(G')|/2+1).$
Thus the result follows.
\end{proof}

\begin{lem}\label{lem:1}Let $\partial (X)$ be a non-trivial 3-cut of a cubic brick $G$ and $|X|\geq 5$.
If $G$ has a matching $M$ of size at least $ (|X|-3)/8+1$ such that
$M\subset E(G[X])$ and every edge in $M$ is removable in $G$, then $G$ has a matching $M'$ of size at least  $|V(G)|/8$ such that   every edge of $M'$ is removable in $G$.
\end{lem}

\begin{proof}
We prove the result by induction on $|V(G)|$. As $|X|\geq 5$, $|M|\geq 2$. So the result holds when $|V(G)|\leq 16$ (note that $K_4$, $\overline{C_6}$ and $R_8$ do not have 3-cuts satisfying the hypothesis).
Assume that the result holds  when $ |V(G)|< n$ $(n> 16)$. We now assume that $|V(G)|=n$. Since $\partial (X)$ is non-trivial, $|\overline{X}|\geq 3$.
We assume first that $|\overline{X}|\leq 5$.  Then $n\leq |X|+5$.  Since $ |M|\geq (|X|-3)/8+1=(|X|+5)/8\geq n/8$, the assertion holds.
So we assume that $|\overline{X}|\geq 7$.
By Proposition \ref{pro:e}, $G^{\Delta}(\overline{X})$ is a brick.
Recalling $|X|\geq 5$, $10\leq |G^{\Delta}(\overline{X})|=|\overline{X}|+3<n$.
 By the induction hypothesis,
$G^{\Delta}(\overline{X})$ has a matching $M_1$ of size at least  $(|\overline{X}|+3)/8$ such that
 every edge of $M_1$ is removable in $G^{\Delta}(\overline{X})$. By Proposition \ref{pro:e3}, every edge in $\partial_{G^{\Delta}(\overline{X})}( \overline{X})$ is not removable.
Then
 $|M_1\cap  G[\overline{X}]|\geq (|\overline{X}|+3)/8-1$ ($-1$  comes from the fact that  at most one edge in the replacement-triangle of $G^{\Delta}(\overline{X})$ lies in $M_1$). Let $M'=M\cup (M_1\cap  G[\overline{X}])$.
Since $M\cap M_1=\emptyset$,  $ |M'|\geq(|X|-3)/8+1+(|\overline{X}|+3)/8-1
=n/8$, the lemma follows by Proposition \ref{pro:e2}.
\end{proof}

{\it The proof of Theorem \ref{thm:main}:}
 Note that $K_4$ is the only cubic brick with 4 vertices, $\overline{C_6}$ is the only cubic brick with 6 vertices, and $R_8$  and the m\"obius ladder with 8 vertices (see Fig. \ref{fig:re}(b), with 4 removable edges)  are all the  cubic bricks with 8 vertices.
 We may assume that $|V(G)|\geq 10$.
By  Lemmas  \ref{lem:nnb1} and \ref{lem:1}, the result holds when $G\in \mathscr{G}$. So we consider the case when $G\notin \mathscr{G}$.
 Assume that $G$ has  a
 3-cut $\partial (X)$  satisfying $|X|\geq 7$, $G/\overline X\in \mathscr{G}$, where $X$ is the minority of $\partial (X)$.
   Then  $G^{\Delta}(X)\in \mathscr{G}$.
  Therefore, $|G^{\Delta}(X)|\geq 10$ and $G^{\Delta}(X)$  has a  non-trivial 3-cut $\partial({Y})$ such that $|Y|\geq 5$, $Y\subset X$ and $G^{\Delta}(X)[Y]$
has a matching of size at least $ (|Y|-3)/8+1$, each edge of which is removable by Lemma \ref{lem:nnb1}. Since $Y\subset V(G)$ and $\partial({Y})$ is also a 3-cut of $G$,
$G^{\Delta}(X)[Y]  \cong G[Y] $.  Thus the result holds by Lemma \ref{lem:1}. Now we   assume that every maximal 3-cut $\partial (X)$  (if exists) satisfies $|X|\leq 5$, or $G$ does not have 3-cut $\partial (X)$ such that $G/\overline X\in \mathscr{G}$. Note that if, for every 3-cut $C$ of $G$, the $C$-contractions are not in
 $\mathscr{G}$, then  $G$ has a  non-trivial 3-cut $C'$ such that one of the $C'$-contractions  is essentially 4-edge-connected other than $K_4$.
By  Lemma  \ref{lem:nnb2} and  Lemma \ref{lem:1}, the theorem follows.
$\hfill\square$

It should be noted that the lower bound in Theorem \ref{thm:main} is not tight for large $|V(G)|$. We do not know the attainable lower bound of independent removable edges of cubic bricks with any number of  vertices. 

\bibliographystyle{plain}
\bibliography{references}








\end{document}